\documentclass[11pt]{amsart}
\usepackage{listings}
\usepackage{caption}
\usepackage{multirow}

\newcommand\D{\text{-}\mathcal{D}}
\newcommand\NTD{\text{-}\mathcal{NTD}}

\DeclareCaptionFormat{listing_caption}{
	\hspace{-0.29cm}
	\fcolorbox{black}{lightgrey}{\parbox{\dimexpr\captionwidth}{\strut \hspace{5pt}#1#2- #3}}
}

\usepackage{amsmath}
\usepackage{amsfonts}
\usepackage{amssymb}
\usepackage{amsthm}
\usepackage{animate}					
\usepackage[toc,page,titletoc]{appendix}			
\usepackage{caption}					
\usepackage{colortbl}					
\usepackage{dsfont}						
\usepackage{empheq}						
\usepackage{enumitem}					
\usepackage{environ}					
\usepackage{float}								
\usepackage{framed}						

\usepackage[utf8x]{inputenc}

\usepackage{fullpage}					
\usepackage{hyperref}					
\usepackage{listings}					
\usepackage{lscape}					
\usepackage{mathabx}					
\usepackage{pgf}						
\usepackage{subcaption}					
\usepackage{xcolor}						

\usepackage[all,cmtip,2cell]{xy}		

\usepackage{etex}						

\usepackage{pgfplotstable}				
\usepackage{booktabs}					

\pgfplotstableset{
	every head row/.style={before row=\toprule,after row=\midrule},
	every last row/.style={after row=\bottomrule}
}

\definecolor{lightgrey}{rgb}{0.97,0.97,0.97}
\definecolor{lesslightgrey}{rgb}{0.9,0.9,0.9}

\definecolor{red}{rgb}{1., 0., 0.}

\definecolor{green}{rgb}{0., 1., 0.}
\definecolor{darkgreen}{cmyk}{100,0,100,0}

\definecolor{babyblue}{cmyk}{0.196,0,0,0}
\definecolor{azureblue}{cmyk}{0.667,0.333,0,0.4}
\definecolor{midnightblue}{cmyk}{1,0.5,0,0.6}
\definecolor{lightblue}{rgb}{.95, .95, 1.}
\definecolor{darkblue}{rgb}{0.1, 0.1, 0.325}
\definecolor{blue}{rgb}{0., 0., 1.}

\usepackage{caption}

\DeclareCaptionFormat{listing_caption}{
	\hspace{-0.29cm}
	\fcolorbox{black}{lightgrey}{\parbox{\dimexpr\captionwidth}{\strut \hspace{5pt}#1#2- #3}}
}

\lstnewenvironment{cpp}[1][]{%
	\lstset{language=C++,#1}}
	{}

\lstnewenvironment{algorithm}[1][]{%
	\lstset{basicstyle=\footnotesize\normalfont, #1}}
	{}

\usepackage{tikz}
\usetikzlibrary[positioning]
\usetikzlibrary{arrows, calc, patterns, positioning}
\usetikzlibrary{backgrounds}	

\renewcommand{\P}		{\mathbb{P}}

\def\Cpp{C{}\texttt{++}}

\newtheoremstyle{defaulttheoremstyle}
	{}	
	{}	
	{}	
	{}	
	{\itshape}	
	{:\\}	
	{.5em}	
	{}	

\theoremstyle{defaulttheoremstyle}

\newtheorem{corollary}{Corollary}[section]

\newtheorem{definition}{Definition}[section]
\newtheorem{example}{Example}[section]

\newtheorem{proof}{Proof}[section]

\newtheorem{proposition}{Proposition}[section]
\newtheorem{remark}{Remark}[section]

\newtheorem{question}{Question}[section]

\renewenvironment{leftbar}{%
  \MakeFramed {\advance\hsize-\width \FrameRestore}}%
{\endMakeFramed}

\NewEnviron{bcorollary}[1]{
\begin{corollary}[#1]%
	\begin{leftbar}%
		\BODY%
	\end{leftbar}%
\end{corollary}%
}

\setlist[enumerate]{topsep=0pt,itemsep=-1ex,partopsep=1ex,parsep=1ex}

\newcommand\mathheader[2]	{\texorpdfstring{$\boldsymbol{#1}$}{#2}}

\newlength\mytemplen
\newsavebox\mytempbox

\makeatletter
\newcommand\mybluebox{%
    \@ifnextchar[
       {\@mybluebox}%
       {\@mybluebox[0pt]}}

\def\@mybluebox[#1]{%
    \@ifnextchar[
       {\@@mybluebox[#1]}%
       {\@@mybluebox[#1][0pt]}}

\def\@@mybluebox[#1][#2]#3{
    \sbox\mytempbox{#3}%
    \mytemplen\ht\mytempbox%
    \advance\mytemplen #1\relax%
    \ht\mytempbox\mytemplen%
    \mytemplen\dp\mytempbox%
    \advance\mytemplen #2\relax%
    \dp\mytempbox\mytemplen%
	\fcolorbox{darkblue}{lightblue}{\hspace{1em}\usebox{\mytempbox}\hspace{1em}}%
}

\makeatother

\NewEnviron{bequation}{
\begin{empheq}[box={\mybluebox[1pt][1pt]}]{equation}
	\BODY
\end{empheq}
}

\usepackage{graphicx}
\usepackage{etoolbox}
\usepackage{ifthen}
\usepackage{blindtext}
\usepackage{graphicx}
\usepackage{eso-pic}

\newcommand\simplexfigure[2]{%
	\begin{subfigure}{.33\textwidth}%
		\centering%
		\includegraphics[width=\linewidth]{#1/complex_#2.pdf}%
		\caption{#2\ifnumcomp{#2}{=}{1}{st}{th} step}%
	\end{subfigure}%
}

\newcommand\simplexanimation[3]{
	\pgfmathparse{int(round(( (#3-#2)/2 ))}
	\animategraphics[controls,height=.4\textheight]{\pgfmathresult}{#1/complex_}{#2}{#3}
}

\newcommand\barcodefigure[5]{
	\begin{figure}[H]%
		\centering%
		\ifdefined\activateanimations
			\simplexanimation{#1}{#2}{#3}
		\else
			\begin{minipage}[t][.2\textheight][t]{\textwidth}%
				\simplexfigure{#1}{#2}%
				\pgfmathparse{int(round( #2+(#3-#2)/5 ))}%
				\simplexfigure{#1}{\pgfmathresult}%
				\pgfmathparse{int(round( #2+2*(#3-#2)/5 ))}%
				\simplexfigure{#1}{\pgfmathresult}%
			\end{minipage}
			\begin{minipage}[t][.2\textheight][t]{\textwidth}
				\pgfmathparse{int(round( #2+3*(#3-#2)/5 ))}%
				\simplexfigure{#1}{\pgfmathresult}%
				\pgfmathparse{int(round( #2+4*(#3-#2)/5 ))}%
				\simplexfigure{#1}{\pgfmathresult}%
				\simplexfigure{#1}{#3}%
			\end{minipage}
		\fi
		\begin{subfigure}[t][.55\textheight][t]{\textwidth}%
			\centering%
			\includegraphics[height=.525\textheight]{#1/barcode.pdf}%
			\caption{Barcode}%
		\end{subfigure}
		\caption{#4}
		\label{#5}
	\end{figure}
}

\newcommand\barcodestatstable[7]{
	\begin{table}[H]
		\centering
		\pgfplotstablevertcat{\pointcloudstats}{#3}
		\pgfplotstablevertcat{\pointcloudstats}{#4}
		\pgfplotstablevertcat{\pointcloudstats}{#5}
		\pgfmathparse{int(#2-1)}
		\pgfplotstabletypeset[
			col sep=tab,
			/pgf/number format/fixed, precision=2,
		    every head row/.style={
		        before row={\specialrule{1.5pt}{0pt}{0pt}\rowcolor{lesslightgrey}},
		        after row={\specialrule{1.5pt}{0pt}{0pt}}
		    },
		    every last row/.style={after row={\specialrule{1.5pt}{0pt}{0pt}}},
		    every nth row={#2[+\pgfmathresult]}{after row=\specialrule{1pt}{0pt}{0pt}},
		    after row={\cline{2-6}},
			columns={distance,0,1,2,3,4},
		    create on use/distance/.style={
		        create col/set list={#1}
		    },
		    columns/distance/.style={
		    	column name={Distance},
		    	string type,
		    	column type=|c||
		    },
			columns/0/.style={column name={$H_i$}, column type=c|},
			columns/1/.style={column name={No. of bars}, column type=c},
			columns/2/.style={column name={Avr. bar size}, column type=c},
			columns/3/.style={column name={\normalsize{Min. bar size}}, column type=>{\footnotesize}c},
			columns/4/.style={column name={\normalsize{Max. bar size}}, column type=>{\footnotesize}c|}
		]{\pointcloudstats}
		\caption{#6}
		\label{#7}
	\end{table}
}

\begin{document}

\author[Scott Balchin and Etienne Pillin]{Scott Balchin and Etienne Pillin}
\address{Department of Mathematics\\
University of Leicester\\
University Road, Leicester LE1 7RH, England, United Kingdom}

\title[TDA on Arbitrary Spaces]{Comparing Metrics on Arbitrary Spaces using Topological Data Analysis}

\begin{abstract}
	We use the notion of topological data analysis to compare metrics on data sets.  We provide two different motivating examples for this.  The first of these is a point cloud data set that has $\mathbb{R}^2$ as its ambient space, and is therefore very visual. the second deals with a very abstract space which arises through the study of non-transitive dice.
\end{abstract}

\maketitle


	\section*{Introduction}

Up until now, topological data analysis has mainly been used to study the structure of a data set which is endowed with a distance.  However, to the authors' knowledge, these tools have not been used to study how the structure of data can change when the metric is varied.  This is mainly due to the fact that the main application of topological data analysis lies in computer vision, where the data is naturally given the Euclidean metric (see \cite{barcode} for an overview of the current state of the theory).  It is only when the data in question has no canonical metric or ambient space that one would have to compare how a change in metric can change the structure.

We begin by briefly introducing the main tools of topological data analysis, before applying it to a point cloud in $\mathbb{R}^2$ which we generate.  We then endow this data with three different metrics, namely the Euclidean, taxicab and supremum metrics.  By studying the three different barcodes arising, we can see how changing the metric can affect the global structure of the data.  We present this along with images of how the simplicial complex develops in each metric, giving some visual intuition on what the change in metric is doing.

After this visual example we move on to a more abstract setting. Non-transitive dice are a relatively underdeveloped problem, a handful of people have provided an insight into this understandable, yet highly complex setting (for example see \cite{SchaeferSchweig2012}, \cite{gardner}, \cite{paradox}).  We will introduce the basic requirements to understand problem, making the example more approachable.  There is no canonical metric to put on the space that arises, therefore we suggest a handful and compare the results using topological data analysis.

	\section{Topological Data Analysis}

	\subsection{Simplicial Complexes}

The general motto of topological data analysis is to give a new way to study large data sets.  By converting data into a topological space, one can apply algebraic topology tools to it in order to infer results.  Namely, one usually uses persistent homology to construct a so-called barcode which can help identify noisy and persistent features.  For us to be able to carry out these techniques, we require a distance $d$ on the data, for all data points $x$, $y$ and $z$ we have:

\begin{enumerate}
	\item $d(x,y) \geq 0$
	\item $d(x,y) = d(y,x)$
	\item  $d(x,y) + d(y,z) \geq d(x,z)$
\end{enumerate}

Note that with the above distance, we may have a finite number of points which are indistinguishable as they have distance $0$, however the simplicial methods that we will use will encode this. Even though the axioms above only give a pseudometric, most of our considered distances will be a true metric.  We will construct a series of topological spaces from the data set, namely the Rips complexes (see \cite{rips}).  

\begin{definition}
	Given a finite collection of points $\{x_\alpha\}$, the \emph{Rips complex} $\mathcal{R}_\epsilon$ is the abstract simplicial complex whose $k$-simplices correspond to unordered $(k+1)$-tuples of points $\{x_\alpha\}^k_0$ that are pairwise within distance $\epsilon$.
\end{definition}

\begin{remark}
	One could also use the \v{C}ech complex construction to get a simplicial complex. Computationally speaking, the Rips complexes are simpler to code as they are flag complexes, meaning that their structure is entirely determined by their $1$-skeletons.
\end{remark}

As we vary $\epsilon$, there is a series of inclusions:
\[
	\mathcal{R}_0 \subset \mathcal{R}_{\epsilon_1} \subset \cdots \subset \mathcal{R}_{\epsilon_{\text{max}}}
\]

\begin{remark}
	Note that if we do have a collection of $k$ points such that they are indistinguishable with respect to our distance measure, then they will be encoded in a $k$-simplex at $\epsilon=0$. This justifies our reasoning that we do not need to consider a strict metric on our data set.
\end{remark} 

	\subsection{Simplicial and Persistant Homology}

A key tool in algebraic topology is that of homology (see \cite{hatcher2002algebraic} for the main source on algebraic topology).  In short, the $k$-th homology group of a simplicial complex computes the number of $k$-dimensional holes.  When we have a whole series of complexes, it is interesting to see when new holes form and when old ones are killed off.  Persistent homology is the tool that allows us do just this (see \cite{fds139645} for an overview).  We can track individual holes, and then compute at which distance $\epsilon$ the hole dies, or a new hole is born.  We then encode this data into a barcode.  We now formally introduce these ideas.

\begin{definition}
	Let $S$ be a simplicial complex, a \emph{simplicial $k$-chain} is a linear sum of $k$-simplices
	\[
		\sum^N_{i=1} c_i \sigma^i
	\]
	where $c_i \in \mathbb{Z}$ and $\sigma^i \in S$ is the $i$-th $k$-simplex.  The group of $k$-chains on $S$ is a free abelian group with basis being all of the $k$-simplices. We denote it $C_k$.
\end{definition}

\begin{definition}
	Given a $k$-simplex $\sigma =   \langle v^0, v^1 , \dots , v^k \rangle$, where the $v^i$ are vertices, there is a \emph{boundary operator} $\partial_k : C_k \to C_{k-1}$ which is defined by
	\[
		\partial_k(\sigma) = \sum^k_{i=0} (-1)^i \langle v^0 , \dots , \hat{v^i}, \dots , v^k \rangle
	\]
	where the hat indicates that we delete that vertex.
\end{definition}

\begin{definition}
	The \emph{$k$-th homology group of $S$} is defined to be the quotient
	\[
		H_k(S) = \text{ker}(\partial_k) / \text{im}(\partial_{k+1})
	\]
	Note that this definition is valid as it can be shown that $(C_k,\partial_k)$ is a chain complex.  Finally, we define the \emph{Betti numbers} of $S$ to be
	\[
		\beta_k = \text{rank}\big(H_k(S)\big)
	\]
\end{definition}

Finding the Betti numbers of our simplicial complexes already tells us a lot about the general structure of the data. However, what we really need to know is how the homology changes when we add new simplices, specifically tracking the behaviour of individual holes.  This is achieved by using persistent homology.  We let $\mathcal{R}$ denote our series of Rips complexes.

\begin{definition}
	For $i < j$, the $(i,j)$-persistent homology of $\mathcal{R}$ is the image of the induced homomorphism in homology $H_\ast(\mathcal{R}^i) \to H_\ast(\mathcal{R}^j)$.
\end{definition}

We will track the Betti numbers through persistent homology, which will tell us in which $\epsilon$ interval each Betti number exists.  These intervals are what we are interested in. When plotted, they indicate which holes are noise and which ones are actually relevant to the structure of the data.  We plot it into a barcode, an example of one is given in figure \ref{fig:barcode_example}.

\begin{figure}[ht]
	\centering
	\includegraphics[scale=0.8]{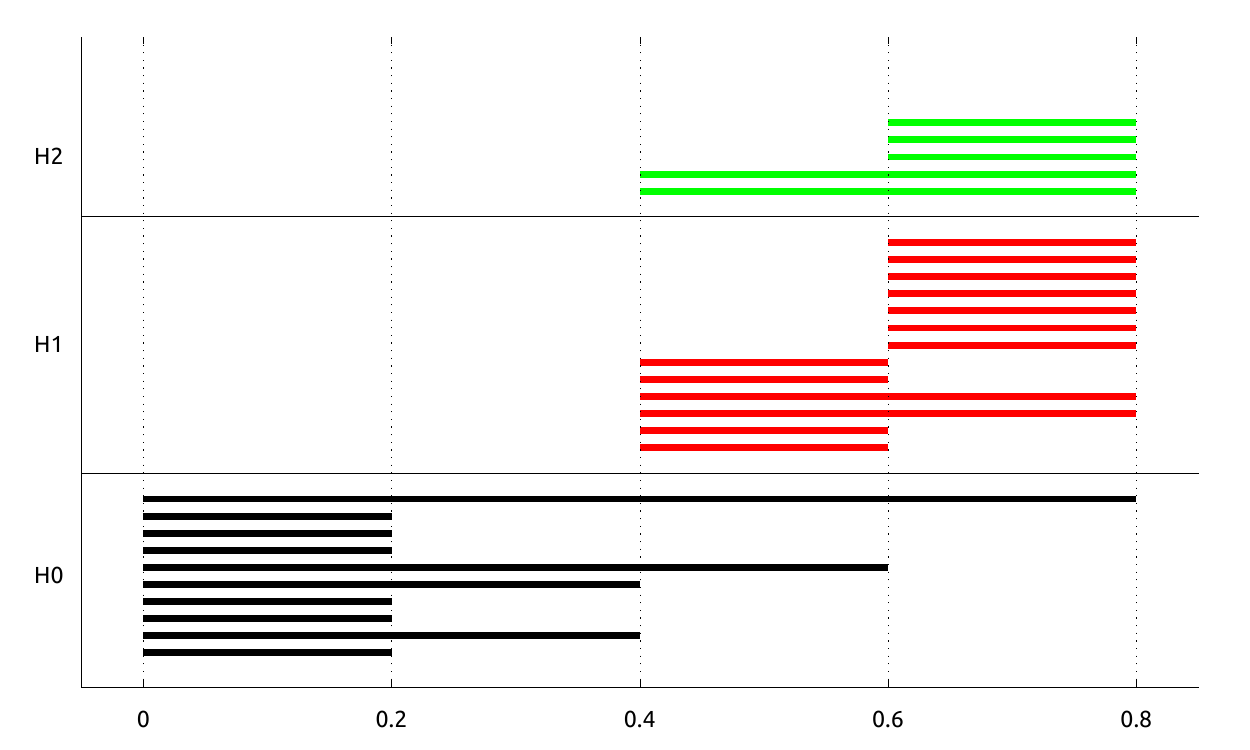}
	\caption{An example of a barcode - The $x$ axis gives the value of $\epsilon$, and the $y$ axis gives the homology degree.  A bar represents the lifespan of a hole in a given homology degree.}
	\label{fig:barcode_example}
\end{figure}

	\subsection{Comparing Metrics With Topological Data Analysis}

The main outlook of this paper is to show that we can use topological data analysis to compare different metrics on a single data space.  We will fix a data set $\mathcal{X}$, which we will endow with a collection of metrics $\{M_i\}_{i \in I}$. We wish to introduce a qualitative measure between metrics.  We can do this by performing topological data analysis on $\mathcal{X}$ with respect to all of the metrics, then compare the resulting barcodes.

It should be noted at this point that there are distances between persistence diagrams, namely the bottleneck and Wasserstein distance \cite{wasser}. This would give us a quantitative measure of the difference between the metrics. In the authors' opinion, this is better suited to calculate the error associated to small perturbations of the data rather than a change in metric.

Instead of using the distances been barcodes, we will give numerical results for each homology class of each metric.  Once we have these numbers we can compare them among the other metrics to get a feel for how similar they are.  The statistics that we will be interested in are the average lifespan of a hole and the number of holes in each dimension. 

\begin{definition}
	Suppose that we have a barcode associated to a data set.  Without loss of generality, we will assume that $\epsilon$ ranges between 0 and 1 (we can do so by diving all distances by the maximum distance). Then:
	\begin{itemize}
		\item[-] The expected lifespan of a hole in dimension $n$ is the average length of the $n$-dimensional bars.
		\item[-] The number of holes in dimension $n$ is the number of bars which are born in the entire span of $\epsilon$.
	\end{itemize}
\end{definition}

	\subsection{Algorithms}

In the above sections we have covered the basic theory of topological data analysis, however this did not give much insight into how one could calculate such invariants of a data set. We represent the distance between the data points as a distance matrix, which by construction is positive valued and symmetric. The global structure of our algorithm is as follows.

\begin{algorithm}[caption=Structure of the code,mathescape]
Build the data set $E$ we want to analyze
Calculate the chosen distance matrix over this space
Make an ordered list of the distance thresholds (from lowest to highest)
For all $\epsilon$ in that list,
	Build the Rips simplicial complex $R_\epsilon$
	Update the simplicial homology and barcode
\end{algorithm}

We will now describe the steps given in the above structure (more can be found on the theory of computational topology in the book of Edelsbrunner and Harer \cite{fds51656}).

	\subsubsection{Building the Rips Simplicial Complex}

Let $N$ be the upper bound on the dimension of the simplicial complex.  We start by building the edges from the list of vertices which are pairwise within a distance below the current threshold $\epsilon$. We then inductively create the list of $(I+1)$-simplices from that of the $I$-simplices, for all $I \in \ldbrack 1, N-1\rdbrack$.

Using an efficient algorithm to do the latter is crucial for the computation to be done within a reasonable amount of time, especially in high dimensions. Let us first consider a naïve algorithm which browses the list of $I$-simplices using  an $(I+1)$-nested iterator \texttt{simplex\_it}. For all combinations of $(I+1)$ $I$-simplices, this iterator checks whether they form a new $(I+1)$-simplex or not. The complexity associated to this step is $n_I^{I+1}$, where $n_I$ is the number of $I$-simplices. Therefore that of the algorithm is
\[
	\sum_{I \in \ldbrack 1, N-1\rdbrack} n_I^{I+1}.
\]

It is not possible to foresee the number of simplices $n_I$. The worst-case scenario, which occurs when $\epsilon \rightarrow \epsilon_\text{max}$, implies that $n_I = {N \choose I+1}$. In particular, assuming $N$ is even, we obtain that the $N/2$ th step of the algorithm is of complexity
\[
	{N \choose I+1}^{N/2}.
\]

Our suggestion for improving this is to avoid browsing simultaneously the same simplices, or the same combinations multiple times. This can be done by ensuring that $\forall i \in \ldbrack 1, I\rdbrack$, \texttt{simplex\_it[i]} $<$ \texttt{simplex\_it[i+1]} . This alone reduces the complexity of the $I$th step to $n_I \choose I+1$. In addition to this, since we build not one simplicial complex, but a series of them, we can store the previous sizes of all the simplex lists at the previous step. We can then make sure that one component of our iterator \texttt{simplex\_it} only browses the $I$-simplices which have just been created.

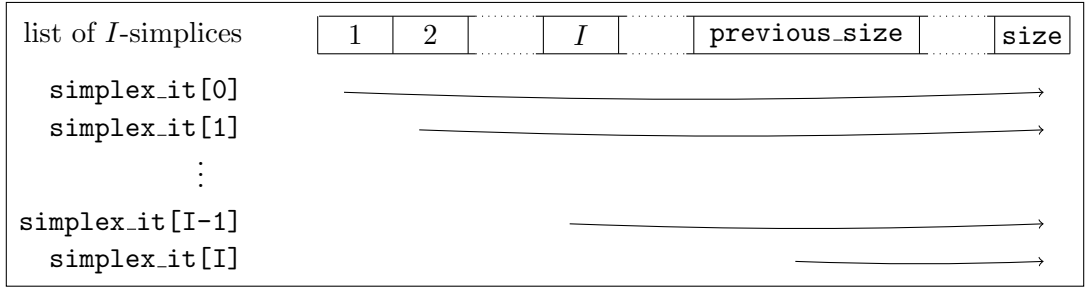
\begin{figure}[ht!]
	\centering
	\begin{tikzpicture}[framed, every node/.style={inner sep=0,outer sep=0}, inner sep = 5pt]
		\node at (0,0) (TL) {};
		\node at (10,0) (TR) {};
		
		\draw[-]
			(TL)--($(TL)+(2.1,0)$)
			($(TL)+(0,-0.5)$)--($(TL)+(2.1,-0.5)$)
			($(TL)+(2.9,0)$)--($(TL)+(4.1,0)$)
			($(TL)+(2.9,-0.5)$)--($(TL)+(4.1,-0.5)$)
			($(TL)+(4.9,0)$)--($(TL)+(8.1,0)$)
			($(TL)+(4.9,-0.5)$)--($(TL)+(8.1,-0.5)$)
			($(TL)+(8.9,0)$)--($(TL)+(10,0)$)
			($(TL)+(8.9,-0.5)$)--($(TL)+(10,-0.5)$);
		\draw[dotted]
			($(TL)+(2.1,0)$)--($(TL)+(2.9,0)$)
			($(TL)+(2.1,-0.5)$)--($(TL)+(2.9,-0.5)$)
			($(TL)+(4.1,0)$)--($(TL)+(4.9,0)$)
			($(TL)+(4.1,-0.5)$)--($(TL)+(4.9,-0.5)$)
			($(TL)+(8.1,0)$)--($(TL)+(8.9,0)$)
			($(TL)+(8.1,-0.5)$)--($(TL)+(8.9,-0.5)$);	
				
		\foreach \i in {0, ..., 5, 8, 9, 10} {
			\draw[-] (\i, 0)--(\i,-0.5);
		}
		
		\node at ($(TL)+(0.5,-0.25)$) {$1$};
		\node at ($(TL)+(1.5,-0.25)$) {$2$};
		\node at ($(TL)+(3.5,-0.25)$) {$I$};
		\node at ($(TL)+(6.5,-0.25)$) {\texttt{previous\_size}};
		\node at ($(TL)+(9.5,-0.25)$) {\texttt{size}};
		
		\node[left=1cm of {$(TL)+(0,-0.25)$}] {list of $I$-simplices};
		\node[left=1cm of {$(TL)+(0,-1)$}] {\texttt{simplex\_it[0]}};
		\node[left=1cm of {$(TL)+(0,-1.5)$}] {\texttt{simplex\_it[1]}};
		\node[left=1.5cm of {$(TL)+(0,-2)$}] {$\vdots$};
		\node[left=1cm of {$(TL)+(0,-2.75)$}] {\texttt{simplex\_it[I-1]}};
		\node[left=1cm of {$(TL)+(0,-3.25)$}] {\texttt{simplex\_it[I]}};
		
		\path[browse/.style={->, shorten <=0.35cm, shorten >=0.35cm, bend right=2}]
			($(TL)+(0,-1)$) edge [browse] node[left] {} ($(TR)+(0,-1)$)
			($(TL)+(1,-1.5)$) edge [browse] node[left] {} ($(TR)+(0,-1.5)$)
			($(TL)+(3,-2.75)$) edge [browse] node[left] {} ($(TR)+(0,-2.75)$)
			($(TL)+(6,-3.25)$) edge [browse] node[left] {} ($(TR)+(0,-3.25)$);
	\end{tikzpicture}
	\caption{A diagram of the above iterator philosophy}
\end{figure}

Finally, out of all those combinations, we only consider those which satisfy the condition
\[
	\forall i \in \ldbrack 1, I-1\rdbrack,\; \texttt{simplex\_it[i]} \text{ shares } \, i \; (I-1)\text{-simplices with }\bigvee_{j<i} \texttt{simplex\_it[j]},
\]

where $\bigvee$ denotes an exclusive union.

For that reason, we numerically define $I$-simplices as their list of $(I-1)$-simplices, and not of vertices.

\begin{algorithm}[caption=Constructing the Rips Complex,mathescape]
For it[0] in the list of $I$-simplices of index previous_size to size
	For it[1] in the list of $I$-simplices of index $1$ to size-1 such that it shares $1$ $(I-1)$-simplex with it[0]
		...
			For it[i] in the list $I$-simplices of index $i$ to size-i such that it shares $i$ $(I-1)$-simplices with $\bigvee_{j<i} \texttt{simplex\_it[j]}$
				...
					For it[I] ...
						Add (it[0],..., it[I]) to the list of $(I+1)$-simplices
\end{algorithm}

The above algorithm implies defining $(I+1)$-times nested loops, for $I \in \ldbrack 1, N-1\rdbrack$. We came up with an elegant implementation of this by developing static loops using template metaprogramming in {\Cpp}.  This way, our program is capable of running for any dimension $D$ provided at compile time without being penalized for it.

	\subsubsection{Computing the Barcodes}

To compute the persistent homology and the barcodes, we introduce the total boundary matrix $M$, which is a means of assembling all the boundary maps $\delta_i$. It is a square matrix of size $\sum_{I \in \ldbrack 0, D\rdbrack} n_I$ and has its values in $\{0, 1\}$. It is defined so that $M_{i,j} = 1$ if and only if
\[
	\left\{
	\begin{aligned}
		&i \text{ and }j \text{ are respectively indices of } I \text{ and } (I+1) \text{-simplices, for some } I	\\
		&\text{ simplex } i \text{ belongs to simplex } {j}
	\end{aligned}
	\right.
\]

Everytime an $I$-simplex is created, it is indexed with respect to the total boundary matrix and $(I+1)$ values are added to the matrix. We then perform a reduction method on the  matrix.  The key point here is that we can take our homology such that the coefficients are in the field $\mathbb{Z}/2 \mathbb{Z}$, which removes any torsion groups that may arise.  This means that when we are reducing the matrix we can make every entry either a 0 or 1 at every step, which makes the computation much easier. 

Another thing to note with the reduction is that we are dealing with sparse matrices.  The original total boundary matrix is sparse as it is created in a combinatorical way, and it indeed stays sparse throughout reduction. Therefore we have coded the entire procedure for sparse matrices, making the calculation more efficient.

We perform the reuction as follows. Take the matrix $M$ and we reduce it to a matrix $R$ such that the the first 1s in all columns appear on different rows.  We do this by adding columns (and then reducing mod 2), we are not permitted to swap columns.  Once we have the matrix $R$ it is possible to see which holes were killed with the addition of a simplex, and if any holes were born.  In particular, adding $\sigma_j$ gives birth to a new homology class if column $j$ of $R$ is zero, and adding $\sigma_j$ kills a homology class if the $j$ column is non-zero in $R$.  If we denote by $R[i,j]$  its lowest 1, then we kill the homology class that was born when $\sigma_i$ was added (see \cite{theoryandprac} for more details).

Using this method, we get an $\epsilon_b$ of when a homology class is born and a $\epsilon_d$ of when the same class dies, if at all.  We can then plot these into a barcode diagram as shown in figure \ref{fig:barcode_example}.

Note that even with the improved simplex building algorithm described above, the complexity of building high-dimensional simplices is extremely high. We can tackle this by picking a maximum dimension $D << N$. An alternative we suggest here is to consider the following stopping criterion.

\begin{definition}[Connectedness stopping criterion]
	We terminate the calculations as soon as the $0$th persistant Betti number is equal to one.
\end{definition}

	\section{Using TDA to Compare Metrics for Low-Dimensional Visual Data}

In this section, we will apply the above philosophy to a data cloud in $\mathbb{R}^2$.   We will be considering the following metrics for points $A=(x_1,y_1)$ and $B=(x_2,y_2)$:
\begin{enumerate}
	\item Euclidean - $d_1(A,B)=\sqrt{(x_1-x_2)^2 + (y_1-y_2)^2}$
	\item Taxicab - $d_2(A,B) = |x_1-x_2| + |y_1 - y_2|$
	\item Supremum - $d_3(A,B)=\text{max}\{|x_1 - x_2| , |y_1 - y_2|\}$
\end{enumerate}

We create a data sample of $n$ points in $\mathbb{R}^2$ by uniformly generating them in a bounded subset parameterised by a collection of curves. Specifically for this example, we created 50 points in a region bounded by a circle with four holes, as shown in figure \ref{fig:point_cloud}.

\begin{figure}[H]
	\centering
	\includegraphics[width=0.4\textwidth]{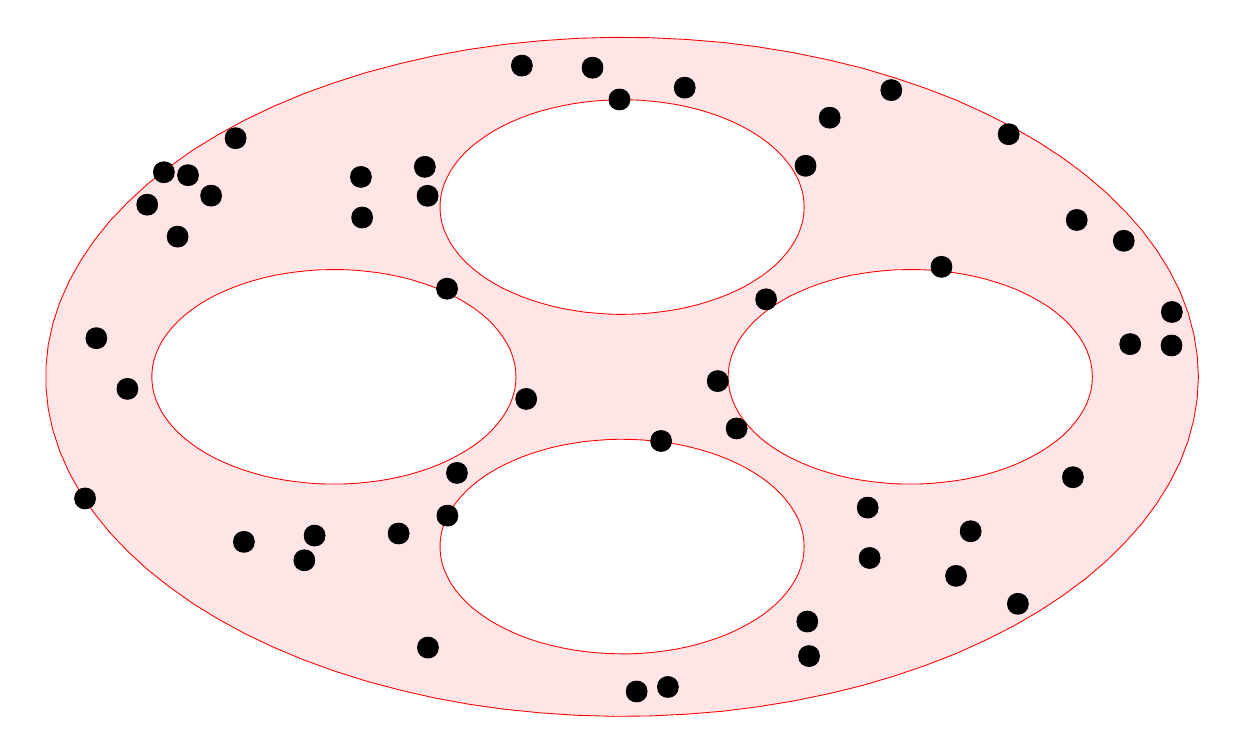}
	\caption{The data set}
	\label{fig:point_cloud}
\end{figure}

\begin{remark}
	As the data is ambient to $\mathbb{R}^2$, we will not consider simplices of dimension greater then $2$ as they will not encode any useful information. 
\end{remark}

The figures on the following pages show the progression of the simplicial complex (restricted to 0,1 and 2-simplices) that we get with respect to the different metrics on this space along with the resultant barcode of the persistent homology.  Note that we have employed the stopping criterion here, so that when the space becomes connected we stop.

We can now perform a visual comparison of the three barcodes. As expected, they are not too different for this example.  The persistent $H_0$ Betti numbers are extremely similar, the only differences appear to occur in the 1 and 2-dimensional holes.  The taxicab has a much steeper increase in the number of $1$-dimensional holes, whereas the other two have a more gradual increase, and also seem to have longer lifespans.  With respect to $H_2$, there is a hole that is picked up by all of the metrics quite early on and persists throughout.  

Table \ref{tab:point_cloud_stats} gives the average number of holes, average lifespan of holes and the minimum/maximum lifespans of holes.

\barcodestatstable{\multirow{3}{*}{Euclidian},,,\multirow{3}{*}{Taxicab},,,\multirow{3}{*}{Supremum},,}{3}
				   {images/point_cloud_50_euclidian/barcode_stats.dat}
				   {images/point_cloud_50_taxicab/barcode_stats.dat}
				   {images/point_cloud_50_supremum/barcode_stats.dat}
				   {Results for the different metrics for the considered cloud of data}
				   {tab:point_cloud_stats}

These figures validate our visual comparisons and also tell us a bit more information.  The average number of bars for the Euclidean and supremum metrics are similar, while in the taxicab metric it follows a different pattern and has a very large value at $H_1$.

We can see that for all three of the metrics, the average lifespan of a bar decreases as we increase the homology dimension. The Euclidean and supremum once again are extremely silimar.

One place where the Euclidean and taxicab metric are similar is in the maximum bar length.  It seems that the longest bar in $H_1$ is detected earlier in the supremum metric.  There are much shorter bars in the taxicab and supremum metric as compared to the Euclidean metric in the higher dimensions. This may be because the data was generated with respect to the Euclidean metric and we should expect less noise because of this.

\newpage
\barcodefigure{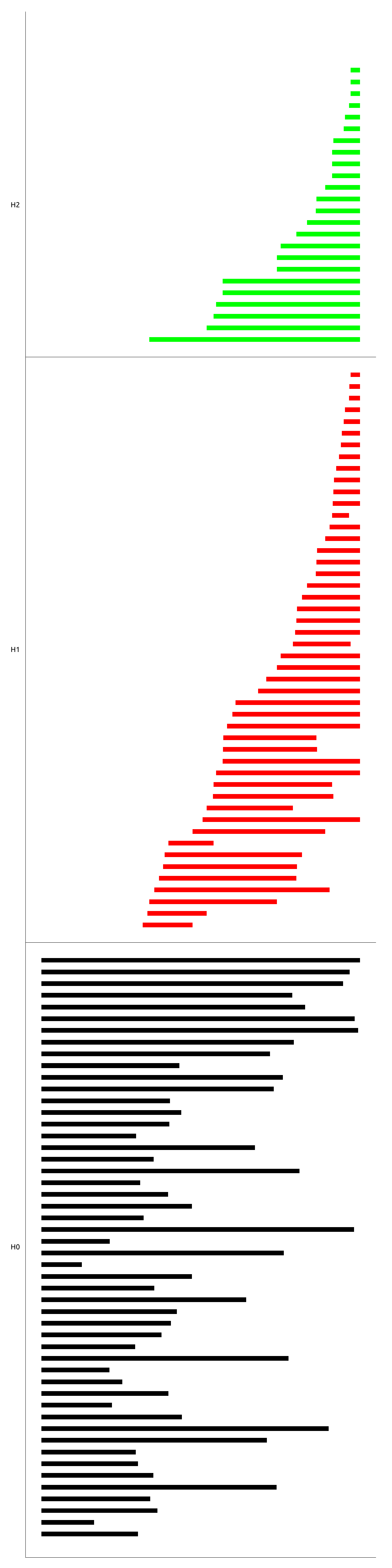}{1}{98}{Results for the Euclidian metric}{fig:point_cloud_euclidian}
\barcodefigure{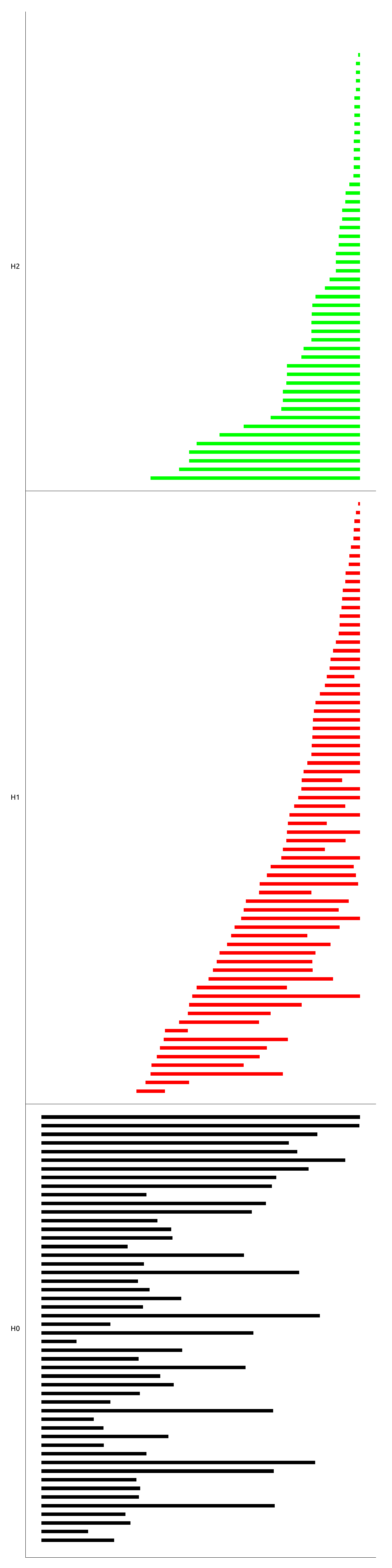}{1}{119}{Results for the taxicab metric}{fig:point_cloud_taxicab}
\barcodefigure{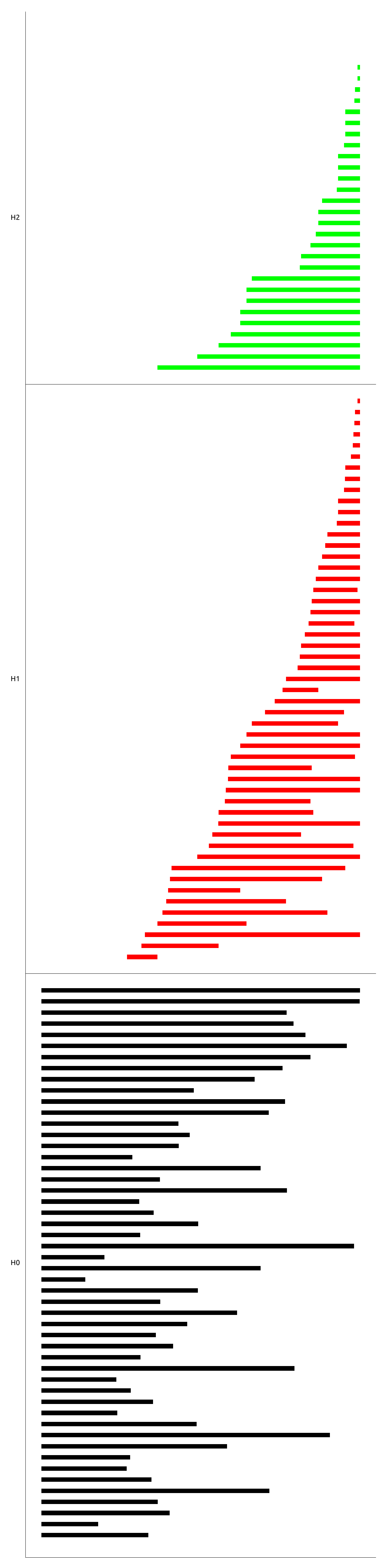}{1}{101}{Results for the supremum metric}{fig:point_cloud_supremum}

	\section{Using TDA to Compare Metrics on Abstract Data}

	\subsection{Non-Transitive Dice}

In this section, we will give a brief overview on the theory of non-transitive dice. This forms a very abstract data set which we will use as an example for comparing metrics on an abstract space.

\begin{definition}
	A \emph{$n$-sided die} with values in the set $K=\ldbrack 1,  k \rdbrack$ is an ordered $n$-tuple $d = [d_1, \dots, d_n]$ where $d_i \in K$.  The collection of all such dice will be denoted $K \D(n)$.
\end{definition}

\begin{example}
	The standard $6$-sided die would be represented by $[1,2,3,4,5,6]$ and is an element of $\ldbrack 1, 6 \rdbrack \D(6)$.
\end{example}

We will only be considering a very special set of dice, namely those in $\ldbrack 1, 6 \rdbrack \D(6)$ where the sum of the number of their faces is 21, akin to the example above. We will denote this set of dice as $\mathcal{DT}(6)$.  However, everything we do in this article works in the more general case of $K \D(n)$.

\begin{definition}[Beating relations]
	Given two dice $X$ and $Y$ in $K\D(n)$, we say that $Y$ \emph{beats} $X$ if $\P(Y > X) > \frac{1}{2}$. We denote this $Y \gg X$.
\end{definition}

Dice endowed with such a beating relation have interesting properties that we wish to study.  Namely, there exists cycles of beating relations, which we call non-transitive dice. We formally outline this below.

\begin{definition}[Cycles of non-transitive dice]
	A \emph{cycle of length $r$ of non-transitive dice} is a ordered collection of dice $(X_1, \dots , X_r) \in K\D(n)$ such that:
	\begin{enumerate}
		\item $X_i \gg X_{i+1}, \,\forall 1 \leq i \leq r-1$.
		\item $X_r \gg X_1$.
	\end{enumerate}
\end{definition}

\begin{example}
	\label{example:noteqav}
	The following diagram
	\begin{center}\begin{tikzpicture}[die/.style={circle,draw,font=\footnotesize}]
		\node[die] (1) at (0, 0)		{$115555$};
		\node[die] (2) at (-1.1547, -2)	{$344444$};
		\node[die] (3) at (1.1547, -2)	{$333366$};
		
		\path[gg/.style={->, shorten <=2pt, shorten >=2pt, bend right=30, font=\footnotesize}]
			(1) edge [gg] node[left] {$\frac{2}{3}$} (2)
			(2) edge [gg] node[below] {$\frac{10}{19}$} (3)
			(3) edge [gg] node[right] {$\frac{5}{9}$} (1);
	\end{tikzpicture}\end{center}
	is a cycle of non-transitive dice in $\mathcal{DT}(6)$ of length $3$. These particular dice are called the Grime dice (see \cite{grime}).
\end{example}

\begin{definition}[Non-transitive dice]
	We say that a die $X \in K \D(n)$ is \emph{non-transitive} if it appears in any non-transitive cycle.  The subset of $K\D(n)$ consisting of all non-transitive dice is denoted $K \NTD(n)$.
\end{definition}

For our particular example of $\mathcal{DT}(6)$, we will denote the subset consisting of all non-transitive dice to be $\mathcal{NTT}(6)$.

\begin{definition}
	Let $X=[x_1, \dots, x_6]$ be a dice in $\mathcal{NTT}(6)$, then its \emph{foliation constant} is given by
	\[
		\mathfrak{f}(X) = x_1-1 + 6-x_6
	\]
\end{definition}

\begin{definition}
Let $X$ be a die in $\mathcal{NTT}(6)$.  We define the \emph{symmetry constant} of $X$, denoted $\mathfrak{s}(X)$ by the following equation:

$$\mathfrak{s}(X) = \sum^{3}_{i=1} \left(  \frac{d_i+d_{n-i}}{2} - \frac{7}{2}  \right)^2$$
\end{definition}

\begin{question}
	\label{ques}
	How can we describe the structure of the space $\mathcal{NTT}(6)$ with respect to the beating relations?
\end{question}

Our main interest will be in answering the above question. We will tackle this approximately using topological data analysis.

	\subsubsection{\mathheader{K\NTD(n)}{K-NTD(n)} as a directed graph}

We compute the space $\mathcal{NTT}(6)$, and then we represent it as a directed graph.  We create a vertex for each die, and then connect node $Y$ to node $X$ with a directed edge if $Y \gg X$. Such a graph gives us lots of information. 

\begin{example}
	Below is the directed graph associated to $\mathcal{NTT}(6)$ along with the associated beating probabilities.
	
	\begin{center}\begin{tikzpicture}[die/.style={circle,draw,font=\footnotesize}]
		\node[die] (1) at ({2*cos(0)},		{2*sin(0)})			{$333336$};
		\node[die] (2) at ({2*cos(360/7)},	{2*sin(360/7)})		{$112566$};
		\node[die] (3) at ({2*cos(2*360/7)}, {2*sin(2*360/7)})	{$144444$};
		\node[die] (4) at ({2*cos(3*360/7)}, {2*sin(3*360/7)})	{$333345$};
		\node[die] (5) at ({2*cos(4*360/7)}, {2*sin(4*360/7)})	{$222366$};
		\node[die] (6) at ({2*cos(5*360/7)}, {2*sin(5*360/7)})	{$114555$};
		\node[die] (7) at ({2*cos(6*360/7)}, {2*sin(6*360/7)})	{$234444$};
		
		\node[die] (8) at (-2.5, -2.5)							{$333444$};
		\node[die] (9) at (2.5, 2.5)							{$122556$};
		\node[die] (10) at (3.5, 3.5)							{$222555$};
		
		\path[gg/.style={->, shorten <=2pt, shorten >=2pt, bend right=12.85, font=\normalsize}]
			(1) edge [gg] node[right] {$\frac{5}{9}$} (2)
			(2) edge [gg] node[above] {$\frac{5}{9}$} (3)
			(3) edge [gg] node[above left] {$\frac{5}{9}$} (4)
			(4) edge [gg] node[left] {$\frac{5}{9}$} (5)
			(5) edge [gg] node[below left] {$\frac{5}{9}$} (6)
			(6) edge [gg] node[below] {$\frac{5}{9}$} (7)
			(7) edge [gg] node[below right] {$\frac{5}{9}$} (1)
			
			(1) edge [gg] node[right] {$\frac{5}{9}$} (9)
			(1) edge [gg] node[right] {$\frac{7}{12}$} (10)
			(9) edge [gg] node[above] {$\frac{5}{9}$} (3)
			(10) edge [gg] node[above] {$\frac{7}{12}$} (3);
		
		\path[gg/.style={->, shorten <=2pt, shorten >=2pt, bend left=12.85, font=\normalsize}]
			(6) edge [gg] node[below] {$\frac{5}{9}$} (8)
			(8) edge [gg] node[left] {$\frac{19}{36}$} (5);
	\end{tikzpicture}\end{center}
	
	We can read off the cycles, for example
	
	\begin{center}\begin{tikzpicture}[align=center, die/.style={circle,draw,font=\footnotesize}]
		\node[die] (1) at (0, 0)		{$114555$};
		\node[die] (2) at (-1.1547, -2)	{$333345$};
		\node[die] (3) at (1.1547, -2)	{$222366$};
		
		\path[gg/.style={->, shorten <=2pt, shorten >=2pt, bend right=30, font=\normalsize}]
			(1) edge [gg] node[left] {$\frac{19}{36}$} (2)
			(2) edge [gg] node[below] {$\frac{5}{9}$} (3)
			(3) edge [gg] node[right] {$\frac{5}{9}$} (1);
	\end{tikzpicture}\end{center}
	
	is a non-transitive cycle.  We computationally find the longest cycle to be a $7$-cycle given by:
	
	\begin{center}\begin{tikzpicture}[die/.style={circle,draw,font=\footnotesize}]
		\node[die] (1) at ({2*cos(0)},		{2*sin(0)})			{$333336$};
		\node[die] (2) at ({2*cos(360/7)},	{2*sin(360/7)})		{$112566$};
		\node[die] (3) at ({2*cos(2*360/7)}, {2*sin(2*360/7)})	{$144444$};
		\node[die] (4) at ({2*cos(3*360/7)}, {2*sin(3*360/7)})	{$333345$};
		\node[die] (5) at ({2*cos(4*360/7)}, {2*sin(4*360/7)})	{$222366$};
		\node[die] (6) at ({2*cos(5*360/7)}, {2*sin(5*360/7)})	{$114555$};
		\node[die] (7) at ({2*cos(6*360/7)}, {2*sin(6*360/7)})	{$234444$};
		
		\path[gg/.style={->, shorten <=2pt, shorten >=2pt, bend right=12.85, font=\normalsize}]
			(1) edge [gg] node[right] {$\frac{5}{9}$} (2)
			(2) edge [gg] node[above] {$\frac{5}{9}$} (3)
			(3) edge [gg] node[above left] {$\frac{5}{9}$} (4)
			(4) edge [gg] node[left] {$\frac{5}{9}$} (5)
			(5) edge [gg] node[below left] {$\frac{5}{7}$} (6)
			(6) edge [gg] node[below] {$\frac{5}{9}$} (7)
			(7) edge [gg] node[below right] {$\frac{5}{9}$} (1);
	\end{tikzpicture}\end{center}
	
	Note that this 7 cycle is not unique, we can replace $[1,1,2,5,6,6]$ by either $[1,1,4,4,5,6]$ or $[2,2,2,5,5,5]$, as seen in the diagram.  This occurs as all three of these dice share the same unique source and target, which limits the length of the longest cycle.
\end{example}

	\subsubsection{Distances on the Non-Transitive Dice Graph}

We can endow a distance matrix on the above graphs, which will then allow us to apply the topological data analysis tools.  Our main metric will be the so-called similarity metric which will tell us how similar two dice are, such as the three interchangeable dice in the above example.

\begin{definition}
	We define the \emph{shortest path} metric on $\mathcal{NTT}(6)$. Given two dice $X$ and $Y$ in $\mathcal{NTT}(6)$, we have
	\[
		d(X,Y) = \text{shortest path from } X \text{ to } Y + \text{shortest path from } Y \text{ to } X.
	\]
	Note that this is well defined as there is always a  path from $X$ to $Y$ by construction.
\end{definition}

\begin{proposition}
	The shortest path distance is indeed a metric on $\mathcal{NTT}(6)$.
\end{proposition}

\begin{proof} \leavevmode
	\begin{enumerate}
		\item[-] $d(X,Y) \geq 0$ is clear as the shortest path between two nodes in a graph is always a positive number.
		\item[-] $d(X,Y) = 0$ if and only if $X=Y$ also holds as if the shortest path between two nodes has length 0 then it means it exactly the same node.
		\item[-] $d(X,Y)=d(Y,X)$ holds from definition.
		\item[-] $d(X,Z) \leq d(X,Y) + d(Y,Z)$ is true as we can always construct a path of equal length via concatenation.
	\end{enumerate}
\end{proof}

\begin{definition}
	We define the \emph{shortest path matrix} of $\mathcal{NTT}(6)$ to be the $n \times n$ matrix $D\big(\mathcal{NTT}(6)\big)$ whose $d_{ij}$ element is exactly $d(X_i,X_j)$.  Here $n$ is the cardinal of $\mathcal{NTT}(6)$, namely 10. 
\end{definition}

Now that we have defined the distance metric, we have a way to compare non-transitive dice.  What we do next is see how similar two dice are by comparing their distances to all other dice.

\begin{definition}
	Let $D=D\big(\mathcal{NTT}(6)\big)$ be the shortest path matrix, we define the $n \times n$ \emph{similarity matrix} $\tilde{D}$ by defining the elements
	\[
		\tilde{D}_{ij}=d_{\mathbb{R}^{n-1}}(\hat{D}_i,\hat{D}_j).
	\]
	Here $d_{\mathbb{R}^{n-1}}$ is the Euclidean metric on $\mathbb{R}^{n-1}$, $\hat{D}_i$ is the $i$-th column of $D$ with the $i$-th value removed (which will always be zero).
\end{definition}

\begin{definition}
	Two dice $X_i$ and $X_j$ in $\mathcal{NTT}(6)$ are said to be \emph{similar} if $\tilde{D}_{ij}=0$.
\end{definition}

\begin{remark}
	Note that this space is no longer a metric as it may not even be Hausdorff. Namely if we have any similar dice the distance between them will be zero.
\end{remark}

	\subsection{Comparing Metrics on $\mathcal{NTT}(6)$}

We previously defined the similarity distance on the space of non-transitive dice.  We will now introduce some other metrics and then compare the resulting barcodes.

Given two dice $X=[x_1, \dots, x_6]$ and $Y=[y_1, \dots, y_6]$:

\begin{enumerate}
	\item (Similarity Distance) $d_1(X,Y) = \tilde{D}_{XY}$
	\item (Eucldean Metric) $d_2(X,Y) = \sqrt{(x_1-y_1)^2 + \cdots + (x_6-y_6)^2}$
	\item (Foliation-Symmetry Distance) $d_3(X,Y) = |(\mathfrak{s}(X) + \mathfrak{f}(X)) - (\mathfrak{s}(Y) + \mathfrak{f}(Y))|$
\end{enumerate} 

Figure \ref{fig:barcode_ntd6_restricted} shows the barcodes associated to the space $\mathcal{NTT}(6)$ with respect to the above three metrics.  As before, we also present some of the relevant statistics that we can obtain from the barcodes in the table \ref{tab:ntd6_stats}.

	\subsection{Interpretation of Results}

The first obvious difference between the abstract space and the sample on $\mathbb{R}^2$ is that we have less critical $\epsilon$ values in the abstract case. This follows by construction as the data has no noise and less nodes.  It is also clear that we made the right decision allowing the dimension of the simplices to be unbounded, there seems to be data encoded even in the 8 and 9-dimensional components of the Rips complex.

Between the three metrics on the non-transitive dice set, there seems to be a connection between the similarity metric and Euclidean metric. It appears that one could make a connection between the two.  This would be an extremely useful thing to have as the Euclidean metric has a much lower computational complexity than the similarity distance. 

On the other hand, the foliation-symmetry distance does not seem to be too useful. There are not many critical values of $\epsilon$ and a large number of holes get born and die simultaneously.  There is also the issue that at $\epsilon=0$ there are not many connected components, this is due to the fact that the way that metric has been constructed implies that there will be many nodes such that the distance between them is zero.

We now utilize table \ref{tab:ntd6_stats} to discuss the problem further.  The first thing to notice is how the number of bars increases then decreases as we increase the dimension. There is a peak at the $H_3$ level, suggesting this is where most information about our data lies. 

The minimum bar sizes of the foliation-symmetry distance are all equal, along with most of the maximum bar sizes being equal.  This suggests that this metric is not very good at ascertaining the importance of holes, and therefore not suited to helping us describe the shape of our data.

In conclusion, it does seem that we could use the Euclidean metric on the dice sets to at least estimate what sort of shape the directed graph with respect to the similarity metric should look like.

\barcodestatstable{\multirow{8}{*}{Similarity},,,,,,,,\multirow{8}{*}{Euclidian},,,,,,,,\multirow{8}{*}{Foliation-Symmetry},,,,,,,}{8}
				   {data/similarity_barcode_stats.dat}
				   {data/euclidian_barcode_stats.dat}
				   {data/balchin_barcode_stats.dat}
				   {Results for the different distances on $\mathcal{NTT}(6)$}
				   {tab:ntd6_stats}

\section{Conclusion}

We have suggested a philosophy that topological data analysis has tools that can be applied to the problem of choosing what is a good distance on an arbitrary space.  Analyzing the results is no easy task, we chose to implement a combination of statistics and visual observations to conclusions.  We believe that there should be an easier and more accurate way to compare barcodes that tell us things about the structure in each dimension.  If such a tool were to be developed, it would most definitely strengthen the theory that has been introduced in this paper.

Our reasoning may also help us justify swapping out the true metric on a space for an approximate one with lower complexity, this would aid in high-end calculations.  In our example, we have noted that the Euclidean distance could be adjusted to approximate the similarity distance on $\mathcal{NTT}(6)$. Not only is this of lower complexity, but it has also allowed us to draw conclusions about abstract spaces such as the non-transitive dice sets.

\begin{figure}[H]
	\centering
	\begin{subfigure}{.33\textwidth}%
		\includegraphics[height=\textheight]{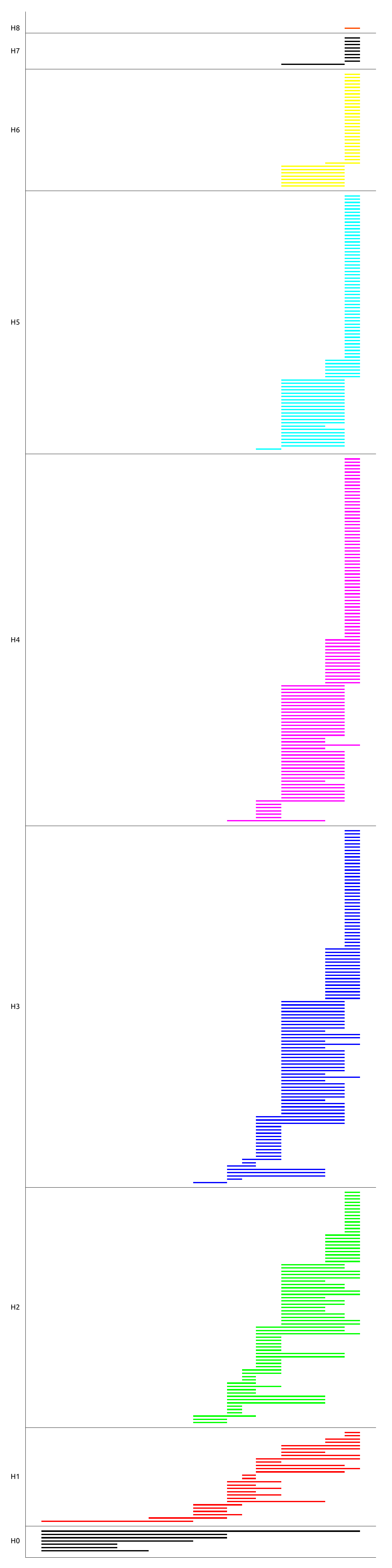}%
		\caption{Similarity metric}%
	\end{subfigure}%
	\begin{subfigure}{.33\textwidth}%
		\includegraphics[height=\textheight]{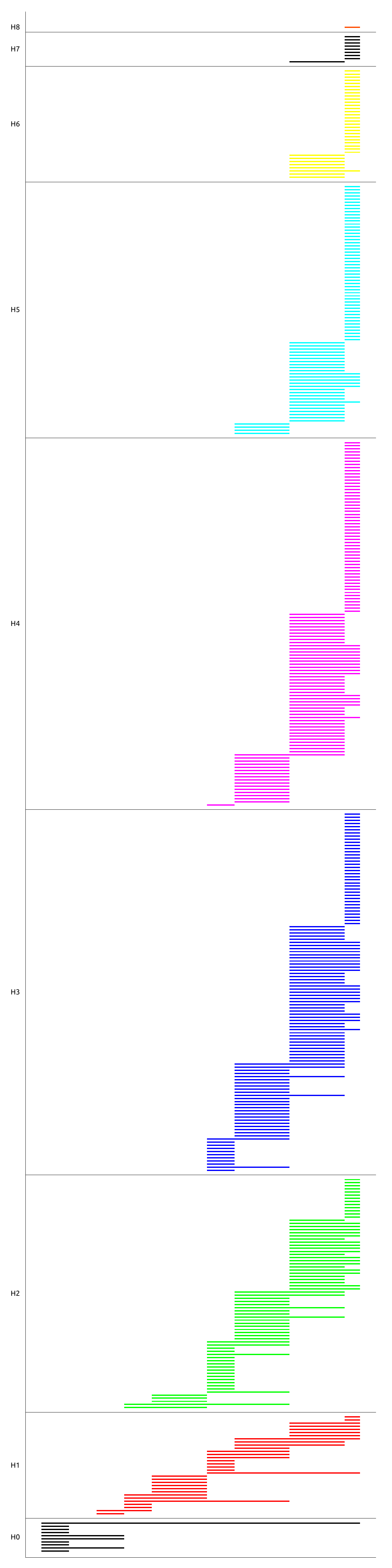}%
		\caption{Euclidian metric}%
	\end{subfigure}%
	\begin{subfigure}{.33\textwidth}%
		\includegraphics[height=\textheight]{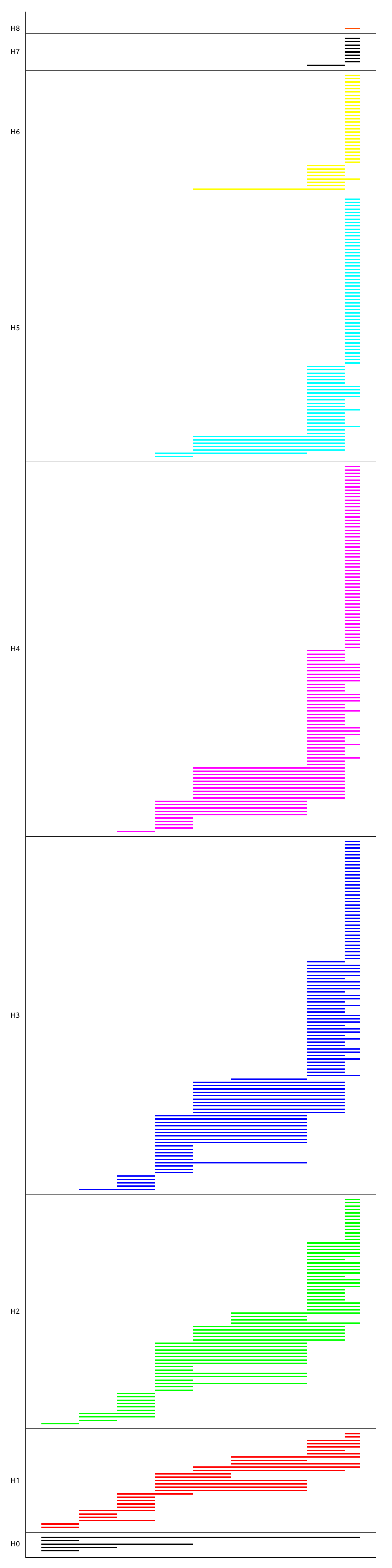}
		\caption{Foliation-Symmetry metric}
	\end{subfigure}%
	\caption{Barcdoes of $\mathcal{NTT}(6)$}
	\label{fig:barcode_ntd6_restricted}
\end{figure}

\bibliographystyle{plain}
\bibliography{tda}
\end{document}